\documentclass[12pt,reqno]{amsart}
\usepackage{graphicx,amscd, color,amsmath,amsfonts,amssymb,geometry,amssymb}
\usepackage[initials]{amsrefs}
\usepackage{courier}
\usepackage[T1]{fontenc}

\newtheorem{theorem}{Theorem}[section]

\newtheorem{proposition}[theorem]{Proposition}
\newtheorem{corollary}[theorem]{Corollary}

\newtheorem{remark}[theorem]{Remark}

\newtheorem{lemma}[theorem]{Lemma}
\newtheorem{problem}[theorem]{Problem}
\newtheorem{definition}[theorem]{Definition}
\numberwithin{equation}{section}

\begin{document}

\title{Sinkhorn-Knopp Theorem for PPT states}

\author[Cariello ]{D. Cariello}

\address{Faculdade de Matem\'atica, \newline\indent Universidade Federal de Uberl\^{a}ndia, \newline\indent 38.400-902 Uberl\^{a}ndia, Brazil.}\email{dcariello@ufu.br}


\keywords{ Sinkhorn, Entanglement, Separability,  PPT states }
\subjclass[2010]{}

\maketitle

\begin{abstract} Given a PPT state $A=\sum_{i=1}^nA_i\otimes B_i \in M_k\otimes M_k$ and a vector $v\in\Im(A)\subset\mathbb{C}^k\otimes\mathbb{C}^k$ with tensor rank $k$, we provide an algorithm that checks whether the positive map $G_A:M_k\rightarrow M_k$, $G_A(X)=\sum_{i=1}^n tr(A_iX)B_i$, is equivalent to a doubly stochastic map. This procedure is based on the search for Perron eigenvectors of completely positive maps and  unique solutions of, at most, $k$ unconstrained quadratic minimization problems.  As a corollary, we can check whether this state can be put in the filter normal form. This normal form is an important tool for studying quantum entanglement.  An extension of this procedure to PPT states in $M_k\otimes M_m$ is also presented. 
\end{abstract}

\section{Introduction}

The separability problem in Quantum Information Theory asks for a  criterion to detect  entanglement of quantum states. Denote by  $M_k$ and $P_k$  the set of complex matrices of order $k$ and the set of positive semidefinite Hermitian matrices of order $k$, respectively.  

This problem is known to be a hard problem \cite{gurvits2003,gurvits2004} even for bipartite mixed states, which 
 can be briefly described as: \\

\begin{quote}
Determine whether a given $A\in M_k\otimes M_m$ is separable or entangled ($A$ is separable if $A=\sum_{i=1}^n C_i\otimes D_i$, where $C_i\in P_k, D_i\in P_m$, for every $i$, and entangled otherwise).\\
\end{quote}

This problem was only solved  for $km\leq 6$. Its original solution requires the classification of positive maps $T:M_k\rightarrow M_m$, $km\leq 6$, which is unknown for $km>6$.
 
Let us identify $M_k\otimes M_m\simeq M_{km}$ via Kronecker product. We shall say that $A=\sum_{i=1}^nA_i\otimes B_i \in M_k\otimes M_m\simeq M_{km}$ is positive under partial transposition, or simply PPT, if  $A$ and its partial transposition $A^{t_2}=\sum_{i=1}^nA_i\otimes B_i^t$ are positive semidefinite Hermitian matrices.

Now, the separable matrices in  $M_k\otimes M_m (km\leq 6)$ are just the PPT matrices \cite{peres, horodeckifamily}.

An alternative proof of this result (not based on positive maps)  was obtained  in $M_2\otimes M_2$ \cite{leinaas}.  Given a positive definite matrix $B\in M_2\otimes M_2\simeq M_4$, there are invertible matrices $R,S\in M_2$ such that 
\begin{equation}\label{eq=filter}
(R\otimes S)B(R\otimes S)^*=\lambda_1\gamma_1\otimes\gamma_1 +\lambda_2\gamma_2\otimes\gamma_2+
\lambda_3\gamma_3\otimes\gamma_3+\lambda_4\gamma_4\otimes\gamma_4,
\end{equation}

\vspace{0.2 cm}\noindent
where $\gamma_1=\frac{1}{\sqrt{2}}Id$, $\gamma_2,\gamma_3,\gamma_4$ are the normalized Pauli matrices and $\lambda_i\in\mathbb{R}$. 
The separability of this $B$ is equivalent  to satisfying the following inequality (see \cite{leinaas}): 

\begin{equation}\label{eq=inequality}
\lambda_1\geq |\lambda_2|+|\lambda_3|+|\lambda_4|.
\end{equation}

\vspace{0.5 cm}

The canonical form presented in  (\ref{eq=filter}) is the so-called filter normal form. 
In general, we say that $A\in M_k\otimes M_m$ can be put in the filter normal form if there are invertible matrices $R\in M_k, S\in M_m$ such that 
$(R\otimes S)A(R\otimes S)^*=\sum_{i=1}^sC_i\otimes D_i$, where $C_1=\frac{1}{\sqrt{k}}Id$, $D_1=\frac{1}{\sqrt{m}}Id$ and $tr(C_iC_j)=tr(D_iD_j)=0$, for every $i\neq j$ \cite{filternormalform, leinaas}. 

Now, the inequality (\ref{eq=inequality}) is a special case of a more general type of inequality (\cite[Theorems 61-62]{carielloQIC}), which provides a sufficient condition for separability. 
Note that  (\ref{eq=inequality}) is also necessary for separability in $M_2\otimes M_2$ and has been obtained from the matrices in the filter normal form. Thus, matrices in this  form
seem to provide sharper inequalities.

Moreover, there are several criteria for detecting entanglement or separability \cite{guhnesurvey}. The filter normal form has also been used to prove the equivalence of some of them \cite{Git}.  Therefore, it plays an important role in Quantum Information Theory.

In this work, we tackle the problem of finding  a necessary and sufficient computable condition for the filter normal form of PPT states.  Our main result is based on a connection  between this form and the Sinkhorn-Knopp theorem for positive maps.

Let $V\in M_k$ be an orthogonal projection and  $VM_kV=\{VXV,\ X\in M_k\}$.
A positive map $T:VM_kV\rightarrow WM_mW$ is a linear transformation such that $T(VM_kV\cap P_k)\subset WM_mW\cap P_m$.
A positive map $T:M_k\rightarrow M_m$ is called doubly stochastic \cite{Landau} if\\ \begin{center}
$T(\frac{Id}{\sqrt{k}})=\frac{Id}{\sqrt{m}}$ and $T^*(\frac{Id}{\sqrt{m}})=\frac{Id}{\sqrt{k}}$,
\end{center} 
\vspace{0.5 cm}\noindent
where $T^*$ is the adjoint of $T$ with respect to the trace inner product: $\langle A,B\rangle=tr(AB^*)$ .

 Two positive maps $L_1:M_k\rightarrow M_m$ and $L_2:M_k\rightarrow M_m$ are said to be  equivalent,  if there are invertible matrices $R_1\in M_k$
and $S_1\in M_m$ such that $L_2(X)=S_1L_1(R_1XR_1^*)S_1^*$.

Now, given a positive semidefinite Hermitian matrix $A=\sum_{i=1}^nA_i\otimes B_i \in M_k\otimes M_m$, consider the positive maps $G_A:M_k\rightarrow M_m$ and  $F_A:M_m\rightarrow M_k$ defined as \\
\begin{center}
$G_A(X)=\sum_{i=1}^n tr(A_iX)B_i$ and $F_A(X)=\sum_{i=1}^n tr(B_iX)A_i$.
\end{center}
\vspace{0.5 cm}

It has been noticed that $A\in M_k\otimes M_m$ can be put in the filter normal form if and only if $G_A:M_k\rightarrow M_m$ (or $G_A((\cdot)^t):M_k\rightarrow M_m$) is equivalent to a doubly stochastic map (e.g., \cite{cariellosink}). 

The next two conditions are known to be necessary and sufficient for the equivalence of a positive map $T:M_k\rightarrow M_m$ with a doubly stochastic one:\\
\begin{itemize}
\item For square maps ($k=m$), the capacity of  $T:M_k\rightarrow M_k$ is positive and achievable \cite{gurvits2004}.
\item For rectangular maps (any $k,m$), $T:M_k\rightarrow M_m$  is equivalent to a positive map with total support \cite[Theorem 3.7]{cariellosink}. \\
\end{itemize}

In the classical Sinkhorn-Knopp theorem \cite{Bapat,Brualdi,Sinkhorn,Sinkhorn2}, the concept of total support plays the key role. The second characterization above adapts this  concept to positive maps.

Some easy properties on $A\in M_k\otimes M_m\simeq M_{km}$ that grant the equivalence of $G_A$ with a doubly stochastic map were obtained in \cite{cariellosink}. 
For example, if $A$ is a positive semidefinite Hermitian matrix such that  either\\

\begin{itemize}
\item $\dim(\ker(A))<\min\{k,m\}\ ($if $k\neq m)$ and  $\dim(\ker(A))<k-1\ ($if $k=m)$ or
\item $G_A(Id)$ and $F_A(Id)$ are invertible and $\dim(\ker(A))<\frac{\max\{k,m\}}{\min\{k,m\}}$.\\
\end{itemize}

Moreover, if $k$ and $m$ are coprime then $G_A:M_k\rightarrow M_m$ is equivalent  to a positive map with total support if and only if it has support \cite[Corollary 3.8]{cariellosink}. 

Now, the positive map $G_A:M_k\rightarrow M_m$ has support if and only if it is fractional-rank non-decreasing \cite[Lemma 2.3]{cariellosink}. A polynomial-time algorithm that checks whether $G_A:M_k\rightarrow M_m$ is fractional-rank non-decreasing was provided in \cite{garg, garg2}. So the equivalence can be efficiently checked when $k$ and $m$ are coprime.

Unfortunately, an algorithm that checks whether the capacity is achievable or the existence of an  equivalent map with total support is unknown (in general).
Our main result is  an algorithm to solve this problem for positive maps $G_A((\cdot)^t):M_k\rightarrow M_m$, where $A\in M_k\otimes M_m$ is a PPT matrix with one extra property  (Algorithm 3). 

Firstly, we consider a PPT matrix $B\in M_k\otimes M_k$ such that there is  $v\in\Im(B)\subset \mathbb{C}^k\otimes\mathbb{C}^k$ with tensor rank $k$. 
There is an invertible matrix $P\in M_k$ such that $\Im(T(X))\supset\Im(X)$, for every $X\in P_k$, where $T(X)=G_C(X^t)$ and $C=(P\otimes Id)B(P\otimes Id)^*$ (Lemma \ref{lemmaP}).\\

A remarkable property owned by every PPT matrix was used in \cite{carielloIEEE} to reduce the separability problem to the weakly irreducible PPT case.  Here, we use the same property (Proposition \ref{propkey}) to show that this positive map $T=G_C(X^t):M_k\rightarrow M_k$ is equivalent to a doubly stochastic map if and only if the following condition holds (Theorem \ref{theoremprincipal}): \\

\begin{quote}
For every orthogonal projection $V\in M_k$ such that $T(VM_kV)\subset VM_kV$  and $T|_{VM_kV}$ is irreducible with spectral radius $\lambda$ (see Definition \ref{definitionIrredPerron}), there is a \textit{\textbf{unique}} orthogonal projection $W\in M_k$ such that $T^*(WM_kW)\subset WM_kW$, $T^*|_{WM_kW}$ is irreducible with spectral radius $\lambda$, $rank(W)= rank(V)$ and
 $\ker(W)\cap\Im(V)=\{\vec{0}\}$.\\
\end{quote}

Our algorithm 3 searches all possible pairs $(V,W)$. This is possible since there are at most $k$ pairs when $B\in M_k\otimes M_k$ is PPT (see Lemma \ref{lemmakey}).

Next, in order to find such $V$ (algorithm 1), we must look for Perron eigenvectors (see Definition \ref{definitionIrredPerron}). 
Given $V$,  the corresponding $W$ is related to the unique zero of  the function $f:M_{k-s\times s}(\mathbb{C})\rightarrow \mathbb{R}^+\cup\{0\}$ defined as $$f(X)=trace\left(T_1^*\left(\begin{pmatrix}Id& X^* \\ 
X& XX^*\end{pmatrix}\right)\begin{pmatrix}X^*X & -X^* \\ 
-X & Id\end{pmatrix}\right),$$
where $T_1(x)=\frac{1}{\lambda}QT(Q^{-1}X(Q^{-1})^*)Q^*$, for a suitable invertible matrix $Q\in M_k$ (Lemma \ref{lemmaQ}).

Note that $f(X)$ seems to be a quartic function. It turns out that $f(X)$ is quadratic, since $T$ is completely positive (check the formula for $f(X)$ in Lemma \ref{solutionW}).  If we regard $M_{k-s\times s}(\mathbb{C})$ as a real vector space then the search for a zero of $f(X)$ is an unconstrained quadratic minimization problem (see Lemma \ref{solutionW} and Remark \ref{remarkquadratic}). 

Since $W$ must be unique then we shall solve this unconstrained quadratic minimization problem only when the solution is known to be unique (Remark \ref{remarkquadratic}). Therefore, the only difficult  part of our algorithm is finding Perron eigenvectors of completely positive maps. 

 There are some algorithms that can be used to find positive definite Hermitian matrices within linear subspaces \cite{HUHTANEN, Zaidi}, however these Perron eigenvectors might be positive semidefinite.\\

Secondly, we extend the result to PPT matrices in $M_k\otimes M_m$ using  \cite[Corollary 3.5]{cariellosink}. Recall the identification $M_m\otimes M_k \simeq M_{mk}$. Let $B=\sum_{i=1}^nC_i\otimes D_i\in M_k\otimes M_m$ be a PPT matrix and define $\widetilde{B}\in M_{mk}\otimes M_{mk}$ as $$\widetilde{B}=\sum_{i=1}^n (Id_{m}\otimes C_i)\otimes (D_i\otimes Id_k).$$ 

 We show that $G_B((\cdot)^t): M_k\rightarrow M_m$ is equivalent to a doubly stochastic map if and only if $G_{\widetilde{B}}((\cdot)^t): M_{mk}\rightarrow M_{mk}$ is equivalent to a doubly stochastic map. Thus, our algorithm 3 can be applied on $\widetilde{B}$ to determine whether $B$ can be put in the filter normal form.\\

 The existence of $v$ with full tensor rank within the range of $A\in M_k\otimes M_k$ is essential for our algorithm to work. There are several possible properties to  impose on a subspace  of $\mathbb{C}^k\otimes \mathbb{C}^k$ in order to guarantee the existence of such $v$   within this subspace \cite{lovasz, meshulam1985, meshulam1989, meshulam2017, pazzis, flanders, dieudonne}. The Edmonds-Rado property \cite{gurvits2003, gurvits2004, ivanyos} is the most relevant to our problem.

Note that, if $D=\sum_{i=1}^nv_i\overline{v_i}^t\in M_k\otimes M_k\simeq M_{k^2}$, where $v_i\in \mathbb{C}^k\otimes\mathbb{C}^k\simeq \mathbb{C}^{k^2}$, then $$G_D(X^t)=\sum_{i=1}^nR_iXR_i^*,$$ where $R_i=F(v_i)^t$ and $F:\mathbb{C}^k\otimes\mathbb{C}^k\rightarrow M_k$ is defined by $F(\sum_{l=1}^t a_i\otimes b_i)=\sum_{l=1}^t a_ib_i^t$. In other words, $D$ is the Choi matrix associated to the completely positive map $G_D(X^t)$ \cite{Choi}. Actually, the map $D\rightarrow F_D((\cdot)^t)$ is the inverse of Jamio{\l}kowski isomorphism \cite{Jamio}. \\

A well known necessary condition for the equivalence of $G_D(X^t)$ with a doubly stochastic map is to be rank non-decreasing, i.e.,  $\text{rank}(G_D(X^t))\geq \text{rank}(X)$ for every $X\in P_k$ \cite{gurvits2003}.
We say that the $\Im(D)=\text{span}\{v_1,\ldots,v_n\}$ has the Edmonds-Rado property, if the existence of $v\in\Im(D)$ with tensor rank $k$ is equivalent to  $G_D(X^t)$ being rank non-decreasing.\\

Now, if $\Im(D)$ is generated by rank 1 tensors then it has the Edmonds-Rado property  \cite{lovasz, gurvits2003}. Therefore the range of every separable matrix has this property \cite{pawel}. Thus, if $D$ is separable with no vector $v$ with full tensor rank in its range then $D$ can not be put in the filter normal form.

We do not know whether $\Im(D)$ has the Edmonds-Rado property when $D$ is only PPT. If the range of any PPT matrix had this property then our algorithm would work for any PPT matrix, since the existence of $v$ would be granted whenever $G_D(X^t)$ is rank non-decreasing.

Finally, there are polynomial-time algorithms that check whether there is $v\in\Im(D)$ with tensor rank $k$  \cite{gurvits2003} and whether $G_D(X^t)$ is rank non-decreasing \cite{garg}.\\

This paper is organized as follows. In Section 2, we present some very important preliminary results. It is worth noticing that Proposition \ref{propkey} is the key that makes our main algorithm (algorithm 3) to work.
This proposition describes the complete reducibility property owned by every PPT matrix \cite{carielloIEEE}. 
In Section 3, the reader can find our main result (Theorem \ref{theoremprincipal}) and one problem derived from this main result (Problem \ref{problem}). This problem can be solved as a unconstrained quadratic minimization problem (Lemma \ref{solutionW}, Remark \ref{remarkquadratic}). In Section 4, we bring all the results together in our algorithms. 
Algorithm 3 provides a way to determine whether a completely positive map originated from a PPT matrix in $M_k\otimes M_k$ (with a rank $k$ tensor within its image) is equivalent to a doubly stochastic map or not. This algorithm can be used to compute the filter normal form of PPT matrices.
In Section 5, we extend our results to PPT matrices in $M_k\otimes M_m$, $k\neq m$. \\

\textbf{Notation:} Let $V\in M_k$ be an orthogonal projection.
Denote by  $VM_k+M_kV=\{VX+YV,\ X,Y\in M_k\}$ and $V^{\perp}=Id-V$, where $V\in M_k$ is an orthogonal projection. Let $\Im(A)$ be the image of the matrix $A$ and $tr(A)$ be its trace. We shall use the trace inner product in $M_k$: $\langle A,B\rangle=tr(AB^*)$.

\newpage

\section{Preliminary Results}

In this section, we present some preliminary results. It is worth noticing that Proposition \ref{propkey} is the key that makes our main algorithm (algorithm 3) to work.
This proposition describes the complete reducibility property  \cite{carielloIEEE}.  \\

\begin{definition}\label{definitionIrredPerron} Let $V,V'\in M_k$ be orthogonal projections. A non-null positive map $T:VM_kV\rightarrow VM_kV$ is called irreducible if $T(V'M_kV')\subset V'M_kV'$ implies that $V=V'$ or $V'=0$. The spectral radius of $T:VM_kV\rightarrow VM_kV$ is the largest absolute value of an eigenvalue of $T:VM_kV\rightarrow VM_kV$. A eigenvector  $\gamma\in P_k$  associated to the spectral radius of $T:VM_kV\rightarrow VM_kV$ is called a Perron eigenvector of $T$.
\end{definition}
\begin{remark} The existence of a Perron eigenvector of an arbitrary positive map $T:VM_kV\rightarrow VM_kV$ is granted by Perron-Frobenius Theory \cite[Theorem 2.3]{evans}.
\end{remark}

\begin{definition}\label{defcompletepositive} A positive map $T:M_k\rightarrow M_k$ is completely positive if $T(X)=\sum_{i=1}^nA_iXA_i^*$, where $A_i\in M_k$ for every $i$. 
\end{definition}

The next three lemmas are well known. We present their proof for the convenience of the reader in Appendix.

\begin{lemma}\label{lemma1} Let $T:M_k\rightarrow M_k$ be a completely positive map. If $V\in M_k$ is an orthogonal projection  such that $T(VM_kV)\subset VM_kV$ then $T(VM_k+M_kV)\subset VM_k+M_kV$.
\end{lemma}

\vspace{0.3 cm}

\begin{lemma}\label{lemma2} Let $T:M_k\rightarrow M_k$ be a positive map. Let $V_1,V\in M_k$ be orthogonal projections such that $\Im(V_1)\subset\Im(V)$ and $T(VM_kV)\subset VM_kV$.
\begin{itemize}
\item[a)] If $T(V_1M_kV_1)\subset V_1M_kV_1$ then $VT^*((V-V_1)M_k(V-V_1)))V\subset (V-V_1)M_k(V-V_1)$.
\item[b)] $T:VM_kV\rightarrow VM_kV$ is irreducible if and only if $VT^*(\cdot)V:VM_kV\rightarrow VM_kV$ is irreducible.
\item[c)] If $\delta\in P_k\cap VM_kV$ is such that $\Im(\delta)=\Im(V)$ and $T(\delta)=\lambda\delta$, $\lambda>0$, then $\lambda$ is the spectral radius of $T|_{VM_kV}$.
\end{itemize}
\end{lemma}

\vspace{0.3 cm}

\begin{lemma}\label{lemma3} Let $T:M_k\rightarrow M_k$ be a completely positive map. Let us assume that $T(VM_kV)\subset VM_kV$, where $V\in M_k$ is an orthogonal projection. Let $\lambda$ be the spectral radius of $T|_{VM_kV}$. Therefore,  $T|_{VM_kV}$ is irreducible if and only if the following conditions hold.
\begin{enumerate}
\item There  are $\gamma,\delta \in VM_kV\cap P_k$ such that $T(\gamma)=\lambda\gamma$, $VT^*(\delta)V=\lambda\delta$ and $\Im(\gamma)=\Im(\delta)=\Im(V)$.
\item The geometric multiplicity of $\lambda$ for $T|_{VM_kV}$ and $VT^*(\cdot)V|_{VM_kV}$ is 1.\\
\end{enumerate}
\end{lemma}

The next proposition is the key that makes our algorithm to  work for PPT matrices. This proposition describes a remarkable property owned by these matrices. The author of \cite{carielloIEEE} used this complete reducibility property to reduce the separability problem in Quantum Information Theory to the weakly irreducible PPT case. \\

\begin{proposition}\label{propkey} $($\textbf{The complete reducibility property of PPT states \cite{carielloIEEE}}$)\\ $Let $B\in M_k\otimes M_k$ be  a PPT matrix. Let  $\{W_1,\ldots, W_s\}\subset  M_k$ be orthogonal projections such that
\begin{itemize}
\item[a)] $W_iW_j=0$, for $i\neq j$,
\item[b)] $\sum_{i=1}^sW_i=Id$,
\item[c)] $W_iM_kW_i$ is left invariant by $G_B((\cdot)^t):M_k\rightarrow M_k$ for every $i$.
\end{itemize}
Then $B=\sum_{i=1}^s(W_i^t\otimes W_i)B(W_i^t\otimes W_i)$.
\end{proposition}
\begin{proof}
Note that $tr(B(W_j^t\otimes W_m))=tr(G_B(W_j^t)W_m)=0$, $ j\neq m$, by item $a)$ and $c)$.\\

Since $B$ and $W_j^t\otimes W_m$ are positive semidefinite Hermitian matrices then \begin{equation}\label{eq=0}
B(W_j^t\otimes W_m)=(W_j^t\otimes W_m)B=0,\ \   j\neq m.\\
\end{equation}

Thus, for $j\neq m$,  $(B(W_j^t\otimes W_m))^{t_2}=(Id\otimes W_m^t)B^{t_2}(W_j^t\otimes Id)=0$ 
and \begin{center}
$tr((Id\otimes W_m^t)B^{t_2}(W_j^t\otimes Id))=tr(B^{t_2}(W_j^t\otimes W_m^{t}))=0$
\end{center}

Since $B$ is PPT then $B^{t_2}$ is positive semidefinite. Thus, \begin{equation}\label{eq=1}
B^{t_2}(W_j^t\otimes W_m^t)=(W_j^t\otimes W_m^t)B^{t_2}=0,\ \   j\neq m.\\
\end{equation}

Now, by item $b)$ and equation \ref{eq=0}, we have $$B=\sum_{i,l,j,m=1}^s(W_i^t\otimes W_l)B(W_j^t\otimes W_m)=\sum_{i,j=1}^s(W_i^t\otimes W_i)B(W_j^t\otimes W_j).$$

By equation \ref{eq=1}, $B^{t_2}=\sum_{i,j=1}^s(W_i^t\otimes W_j^t)B^{t_2}(W_j^t\otimes W_i^t)=\sum_{i=1}^s(W_i^t\otimes W_i^t)B^{t_2}(W_i^t\otimes W_i^t).$\\

Finally, $B=\sum_{i=1}^s(W_i^t\otimes W_i)B(W_i^t\otimes W_i).$
\end{proof}
\vspace{0.5 cm}

The next lemma shall be used  in our main theorem to prove the uniqueness of $W$ (item 3 of \ref{theoremprincipal}). Its proof is another consequence of the complete reducibility property. 

\vspace{0.5 cm}

\begin{lemma}\label{lemmakey} Let $B\in M_k\otimes M_k$ be  a PPT matrix and $\{W_1,\ldots,W_s\}\subset M_k$  be orthogonal projections as in Proposition \ref{propkey}. Suppose that $WM_kW$ is left invariant by $G_B((\cdot)^t):M_k\rightarrow M_k$ and $G_B((\cdot)^t)|_{WM_kW}$ is irreducible. If there is $t\leq s$ such that $G_B((\cdot)^t)|_{W_iM_kW_i}$ is irreducible for every $1\leq i\leq t$ then either
$$W=W_j\ (\text{for some } 1\leq j\leq t)\ \ \ \ or \ \ \ \ \Im(W)\subset \Im(Id-W_1-\ldots-W_t).$$
This result is also valid if $G_B((\cdot)^t):M_k\rightarrow M_k$ is replaced by its adjoint $F_B(\cdot)^t:M_k\rightarrow M_k$.
\end{lemma}
\begin{proof}
By Proposition \ref{propkey}, $G_B(W^t)=\sum_{i=1}^sW_iG_B(W_i^tW^tW_i^t)W_i.$ 
So, $W_iG_B(W^t)=G_B(W^t)W_i$, $\forall i$.\\

Since $W_i$ and $G_B(W^t)$ commute, if  $W_iG_B(W^t)\neq 0$, for some $i\leq t$,   then 
\begin{center}
$\Im(W_i)\cap\Im(G_B(W^t))\neq\{\vec{0}\}$.
\end{center}

 Let $\widetilde{W}$ be the orthogonal projection onto $\Im(W_i)\cap\Im(G_B(W^t))$.

Next,  $\widetilde{W}M_k\widetilde{W}\subset W_iM_kW_i\cap WM_kW$ is left invariant by $G_B((\cdot)^t)$ and \begin{center}
$G_B((\cdot)^t)|_{W_iM_kW_i}$, $G_B((\cdot)^t)|_{WM_kW}$ are irreducible.
\end{center}  

Therefore,  $\widetilde{W}= W_i= W$, for some $i\leq t$.

Now, if $W_iG_B(W^t)= 0$, for every $1\leq i\leq t$, then $\Im(G_B(W^t))\subset\Im(Id-W_1-\ldots-W_t)$. 

Finally, since  $G_B((\cdot)^t)|_{WM_kW}$ is irreducible then $\Im(W)=\Im(G_B(W^t))$.
\end{proof}

\vspace{0.5 cm}

\begin{lemma}\label{lemmapropertyowned} Let $T:M_k\rightarrow M_k$ be a positive map and $Q\in M_k$ an invertible matrix. Let us assume that for every orthogonal projection $V\in M_k$ such that $T(VM_kV)\subset VM_kV$, $T|_{VM_kV}$ is irreducible with spectral radius $\lambda$, there is a unique orthogonal projection $W\in M_k$ such that $T^*(WM_kW)\subset WM_kW$, $T^*|_{WM_kW}$ is irreducible with spectral radius $\lambda$, rank$(V)=$rank$(W)$ and $\ker(W)\cap\Im(V)=\{\vec{0}\}$. Then the same property is owned by $S:M_k\rightarrow M_k$, $S(X)=QT(Q^{-1}X(Q^{-1})^*)Q^*$.
\end{lemma}
\begin{proof} Let $V'$ be an orthogonal projection such that $S(V'M_kV')\subset V'M_kV'$, $S|_{V'M_kV'}$ is irreducible with spectral radius $\lambda$. 

Since $S,T$ are similar maps then $T(VM_kV)\subset VM_kV$, $T|_{VM_kV}$ is irreducible with spectral radius $\lambda$, where $V$ is the orthogonal projection onto $\Im(Q^{-1}V'(Q^{-1})^*)$.\\

Let $W$ be the orthogonal projection described in the statement of this theorem and $W'$ the orthogonal projection  
onto $\Im((Q^{-1})^*WQ^{-1})$. Note that $S^*(X)=(Q^{-1})^*T^*(Q^*XQ)Q^{-1}$. \\

 Therefore,

\begin{itemize}
\item $S^*(W'M_kW')\subset W'M_kW',$
\item $S^*|_{W'M_kW'}$ is irreducible with spectral radius $\lambda$,
\item rank$(W')=$ rank$(W)=$ rank$(V)=$ rank$(V')$.\\
\end{itemize}

Furthermore, since $\ker(W)\cap\Im(V)=\ker(Q^*W'Q)\cap \Im(Q^{-1}V'(Q^{-1})^*)=\{\vec{0}\}$ then $$\ker(W')\cap \Im(V')=\{\vec{0}\}.$$

Next, let $W''$ be another orthogonal projection satisfying the same properties of $W'$. 

Since $T^*(X)=Q^*S^*((Q^{-1})^*XQ^{-1})Q$ then the orthogonal projection onto $\Im(Q^*W''Q)$ satisfies the same properties of $W$ (by the argument above). \\

By the uniqueness of $W$, we have $\Im(W)=\Im(Q^*W''Q)$. Hence, $$\Im(W'')=\Im((Q^{-1})^*W(Q^{-1}))=\Im(W').$$

 So the uniqueness of $W'$ follows.
\end{proof}

\vspace{0.5 cm}

\begin{lemma}\label{lemmaP} Let $B\in M_k\otimes M_k\simeq M_{k^2}$ be a positive semidefinite Hermitian matrix. Suppose there is $v\in\Im(B)\subset \mathbb{C}^k\otimes\mathbb{C}^k$ with tensor rank $k$. There is an invertible matrix $P\in M_k$ such that $\Im(G_A(X^t))\supset\Im(X)$ for every $X\in P_k$, where $A=(P\otimes Id)B(P\otimes Id)^*$.
\end{lemma}
\begin{proof}
Let $P\in M_k$ be an invertible matrix such that $(P\otimes Id)v=u$, where $u=\sum_{i=1}^k e_i\otimes e_i$ and $\{e_1,\ldots, e_k\}$ is the canonical basis of $\mathbb{C}^k$. Thus, $u\in\Im(A)$, where $A=(P\otimes Id)B(P\otimes Id)^*$.

There is $\epsilon>0$ such that $A-\epsilon (uu^t)$ is a positive semidefinite Hermitian matrix. Thus, $G_{A-\epsilon uu^t}(X^t)$ is a positive map.
Note that\begin{center}
 $G_A(X^t)=G_{A-\epsilon uu^t}(X^t)+\epsilon\ G_{uu^t}(X^t)$ and $G_{uu^t}(X^t)=X$.
\end{center} 

Finally, $\Im(G_A(X^t))\supset\Im(X)$ for every $X\in P_k$.
\end{proof}

\section{Main Results}

In this section we present our main result (Theorem \ref{theoremprincipal}) and one problem derived from this main result (Problem \ref{problem}). This problem can be solved as a unconstrained quadratic minimization problem (Lemma \ref{solutionW}, Remark \ref{remarkquadratic}).

\begin{theorem}\label{theoremprincipal} Let $A\in M_k\otimes M_k$ be  a PPT matrix and $T:M_k\rightarrow M_k$ be the completely positive map $G_A((\cdot)^t):M_k\rightarrow M_k$. Suppose that $\Im(T(X))\supset \Im(X)$ for every $X\in P_k$. The following statements are equivalent:
\begin{enumerate}
\item $T:M_k\rightarrow M_k$ is equivalent to a doubly stochastic map.\\

\item There are orthogonal projections $\{V_1,\ldots,V_s\}\subset M_k$ such that  \ \  $\mathbb{C}^k=\bigoplus_{i=1}^s\Im(V_i)$,\\ $T(V_iM_kV_i)\subset V_iM_kV_i$, $T|_{V_iM_kV_i}$ is irreducible for every $i$.\\

\item For every orthogonal projection $V\in M_k$ such that $T(VM_kV)\subset VM_kV$, $T|_{VM_kV}$ is irreducible with spectral radius $\lambda$, there is a unique orthogonal projection $W\in M_k$ such that $T^*(WM_kW)\subset WM_kW$, $T^*|_{WM_kW}$ is irreducible with spectral radius $\lambda$, $rank(W)= rank(V)$ and
 $\ker(W)\cap\Im(V)=\{\vec{0}\}$.

\end{enumerate}
\end{theorem}
\begin{proof}
$(1\Leftrightarrow 2)\ $The existence of the orthogonal projections described in item 2 and the condition $\Im(T(X))\supset \Im(X)$ for every $X\in P_k$ imply that $T:M_k\rightarrow M_k$ has a positive achievable capacity. Then it is equivalent to a doubly stochastic map (See \cite{gurvits2004}). Now, if a positive map $T:M_k\rightarrow M_k$ is equivalent to a doubly stochastic map and $\Im(T(X))\supset \Im(X)$ for every $X\in P_k$ then we can easily find orthogonal projections as described in  item 2. So the first two conditions are equivalent. Check also \cite[Theorem 3.4]{cariellosink} for a different approach  based on Sinkhorn and Knopp original proof.\\\\
$(2\Rightarrow 3)\ $
There is an invertible matrix  $P\in M_k$ such that $\Im(PV_iP^*)\perp\Im(PV_jP^*)\ (i\neq j),$ since $\mathbb{C}^k=\bigoplus_{i=1}^s\Im(V_i)$.

Let $W_i$ be the orthogonal projection onto $\Im(PV_iP^*)$. Hence, $\sum_{i=1}^sW_i=Id$, $W_iW_j=0$ $(i\neq j)$.\\

Define $B=((P^{-1})^t\otimes P)A((P^{-1})^t\otimes P)^*$. Note that $B$ is PPT and $$G_B(X^t)=P(T(P^{-1}X(P^{-1})^*))P^*\text{ and } (F_B(X))^t=(G_B(X^t))^*=(P^{-1})^*T^*(P^*XP)P^{-1}$$ 

Thus, the hypotheses of Lemma \ref{lemmakey} hold for this $B$. Moreover, it follows from the hypothesis of item 2 that
\begin{center}
$G_B((\cdot)^t)|_{W_iM_kW_i}$ is irreducible, $1\leq i\leq s$, and $\Im(G_B(X^t))\supset\Im(X)$ for every $X\in P_k$.

\end{center}

Next, if $T(VM_kV)\subset VM_kV$ and $T|_{VM_kV}$ is irreducible then $G_B((\widetilde{V}M_k\widetilde{V})^t)\subset \widetilde{V}M_k\widetilde{V}$ and $G_B((\cdot)^t)|_{\widetilde{V}M_k\widetilde{V}}$ is irreducible, where $\widetilde{V}$ is the orthogonal projection onto $\Im(PVP^*)$. 

By Lemma \ref{lemmakey}, $\widetilde{V}=W_i$ for some $i$. Hence, \begin{center}
$\Im(PVP^*)=\Im(PV_iP^*)$ and $V=V_i$.
\end{center}

Analogously, if $T^*(WM_kW)\subset WM_kW$ and $T^*|_{WM_kW}$ is irreducible then, by Lemma \ref{lemmakey}, \begin{center}
$\Im((P^{-1})^*WP^{-1})=\Im(W_j)$ for some $j$.
\end{center} 

Note that $\ker(W)\cap\Im(V)=\{\vec{0}\}$ if and only if $\ker(W_j)\cap\Im(PV_iP^*)=\{\vec{0}\}$. \\

Since $\Im(PV_iP^*)=\Im(W_i)$ and $W_lW_m=0$ $(l\neq m)$ then $\ker(W)\cap\Im(V_i)=\{\vec{0}\}$ if and only if  $W$ is the orthogonal projection onto $\Im(P^*W_iP)$ $($i.e. $j=i)$. Thus, there is a unique $W$ such that \begin{center}
$T^*(WM_kW)\subset WM_kW$, $T^*|_{WM_kW}$ is irreducible and $\ker(W)\cap\Im(V)=\{\vec{0}\}$.
\end{center}

Recall that, $\Im(W)=\Im(P^*W_iP)=\Im(P^*PVP^*P)$. So rank$(W)=$ rank$(V)$.\\

Next, $T|_{VM_kV}$ and $G_B((\cdot)^t)|_{W_iM_kW_i}$ have the same spectral radius, since they are similar and 
the same is valid for $T^*|_{WM_kW}$ and $(F_B(\cdot))^t|_{W_iM_kW_i}$. \\

Since  $W_lW_m=0$ for every $l\neq m$ then \begin{center}
$(G_B((\cdot)^t)|_{W_iM_kW_i})^*=(F_B(\cdot))^t|_{W_iM_kW_i}$.
\end{center} 

Thus, $T|_{VM_kV}$ and $T^*|_{WM_kW}$ have also the same spectral radius.\\\\
$(3\Rightarrow 2)\ $ Let $V_1\in M_k$ be an orthogonal projection such that\begin{center}
 $T(V_1M_kV_1)\subset V_1M_kV_1$, $T|_{V_1M_kV_1}$ is irreducible. 
\end{center}

If $V_1=Id$ then the proof is complete. 

If $V_1\neq Id$ then, by hypothesis of item 3, there is an orthogonal projection $W_1$ such that 

\begin{itemize}
\item $T^*(W_1M_kW_1)\subset W_1M_kW_1$,
\item $T^*|_{W_1M_kW_1}$ is irreducible,
\item rank$(W_1)=\ $rank$(V_1)$ and $\ker(W_1)\cap\Im(V_1)=\{\vec{0}\}$.
\end{itemize}

Thus, $\Im(V_1)\oplus\Im(Id-W_1)=\mathbb{C}^k$ and  \begin{center}
$T((Id-W_1)M_k(Id-W_1))\subset (Id-W_1)M_k(Id-W_1)$,
\end{center} by item $a)$  of Lemma \ref{lemma2}.

Let $Q_1$ be an invertible matrix such that \begin{center}
$Q_1V_1=V_1$ and $\Im(Q_1(Id-W_1))=\Im(Id-V_1)$.
\end{center}

Define \begin{center}
$A_1=((Q_1^{-1})^t\otimes Q_1)A((Q_1^{-1})^t\otimes Q_1)^*$ and $T_1(X)=G_{A_1}(X^t)=Q_1T(Q_1^{-1}X(Q_1^{-1})^*)Q_1^*$.
\end{center}

Note that 
$T_1(V_1M_kV_1)\subset V_1M_kV_1$, $T_1|_{V_1M_kV_1}$ is irreducible, \begin{center}
$T_1((Id-V_1)M_k(Id-V_1))\subset (Id-V_1)M_k(Id-V_1)$
\end{center}
 and the same occurs if we replace $T_1$ by $T_1^*$.
 
 Now, let $V_2$ be an orthogonal projection such that \begin{center}
$T_1(V_2M_kV_2)\subset V_2M_kV_2 \subset  (Id-V_1)M_k(Id-V_1)$ and $T_1|_{V_2M_kV_2}$ is irreducible.
\end{center}

If $V_2=Id-V_1$ then $T_1$ satisfies the conditions of item 2. Hence, the same conditions hold for $T$ and the proof is complete. 

Next, assume that $V_2\neq Id-V_1$. By Lemma \ref{lemmapropertyowned}, the property owned by $T$ described in item 3 of the statement of this theorem is also owned by $T_1$. Hence, there is an orthogonal projection $W_2$ such that \begin{itemize}
\item $T_1^*(W_2M_kW_2)\subset W_2M_kW_2$, \item $T_1^*|_{W_2M_kW_2}$ is irreducible, \item rank$(W_2)=$ rank$(V_2)$ and $\ker(W_2)\cap\Im(V_2)=\{\vec{0}\}$.
\end{itemize}

Moreover, since $A_1$ satisfies the hypotheses of Lemma \ref{lemmakey}  with $t=1$ and $s=2$ then \begin{center}
$W_2=V_1$ or $\Im(W_2)\subset \Im(Id-V_1)$. 
\end{center}

However,  $W_2=V_1$ is not possible, since $\Im(V_2)=\ker(V_1)\cap \Im(V_2)$ and $\ker(W_2)\cap\Im(V_2)=\{\vec{0}\}$. 

Therefore, $\Im(W_2)\subset \Im(Id-V_1)$. 

Let $Q_2$ be an invertible matrix such that \begin{center}
$Q_2(V_1+V_2)=V_1+V_2$ and $\Im(Q_2(Id-V_1-W_2))=\Im(Id-V_1-V_2)$.
\end{center}

Define \begin{center}
$A_2=((Q_2^{-1})^t\otimes Q_2)A_1((Q_2^{-1})^t\otimes Q_2)^*$ and $T_2(X)=G_{A_2}(X^t)=Q_2T_1(Q_2^{-1}X(Q_2^{-1})^*)Q_2^*$.
\end{center}

Note that $T_2(V_iM_kV_i)\subset V_iM_kV_i$, $T_2|_{V_iM_kV_i}$ is irreducible for $1\leq i\leq 2$
 and  
 
 \begin{center}
 $T_2((Id-V_1-V_2)M_k(Id-V_1-V_2))\subset (Id-V_1-V_2)M_k(Id-V_1-V_2)$.
\end{center}

Repeating this argument $s$ times $(s\leq k)$, we obtain 
 $T_s(V_iM_kV_i)\subset V_iM_kV_i$, 
 $T_s|_{V_iM_kV_i}$ is irreducible
 for every $1\leq i\leq s$, $V_iV_j=0$ for $i\neq j$ and $\sum_{i=1}^sV_i=Id$, where $T_s(X)=RT(R^{-1}X(R^{-1})^*)R^*$ for some invertible $R\in M_k$. 
\end{proof}

\vspace{0.5 cm}

This last theorem gives rise to the following problem. We present a solution for this problem in the next two lemmas (\ref{lemmaQ}, \ref{solutionW}). 

\vspace{0.5 cm}

\begin{problem}\label{problem} Given a completely positive map $T:M_k\rightarrow M_k$ and an orthogonal projection $V\in M_k$ such that 
$T(VM_kV)\subset VM_kV$  and $T|_{VM_kV}$ is irreducible,
find
 an orthogonal projection $W$ such that 
\begin{itemize}
\item $T^*(WM_kW)\subset WM_kW$,
\item $T^*|_{WM_kW}$ is irreducible,
\item $rank(W)= rank(V)$ and 
\item $\ker(W)\cap\Im(V)=\{\vec{0}\}$.
\end{itemize}

We shall call such $W$ a solution to the Problem \ref{problem} subjected to $T$ and $V$.
\end{problem}

\vspace{0.5 cm}

The next lemma simplifies our search for $W$ of problem \ref{problem} and relates  it  to a solution of a  unconstrained quadratic minimization problem (lemma \ref{solutionW}). Recall that item 3 of  theorem \ref{theoremprincipal} requires the uniqueness of  this solution. Thus, we shall solve this minimization problem only when the uniqueness of the solution is granted (see remark \ref{remarkquadratic}). 
\vspace{0.5 cm}

\begin{lemma}\label{lemmaQ} Let $T:M_k\rightarrow M_k$ and $V\in M_k$ be as in Problem \ref{problem}.
Let $\lambda$ be the spectral radius  of $T|_{VM_kV}$. There is an invertible matrix $Q\in M_k$ such that a solution to the Problem \ref{problem} subjected to $T$ and $V$ is the orthogonal projection onto $\Im(Q^*W_1Q)$, where $W_1$ is a solution to the Problem \ref{problem} subjected to\\
\begin{enumerate}
\item $T_1:M_k\rightarrow M_k,\ T_1(x)=\frac{1}{\lambda}QT(Q^{-1}X(Q^{-1})^*)Q^*$,
\item $V_1=\begin{pmatrix}Id_{s\times s} & 0_{s\times k-s} \\ 
0_{k-s\times s} & 0_{k-s\times k-s}\end{pmatrix}$, where $s= rank(V)$, \\
\item $V_1T_1^*(V_1)V_1=V_1$.
\end{enumerate}
\end{lemma}
\begin{proof} By Lemma \ref{lemma2}, item b), $VT^*(\cdot)V:VM_kV\rightarrow VM_kV$ is irreducible. \\

By Lemma \ref{lemma3}, there is 
$\delta_1\in P_k\cap VM_kV$ such that $VT^*(\delta_1)V=\lambda\delta_1$ and $\Im(\delta_1)=\Im(V)$.\\

Define $R=\delta_1^{\frac{1}{2}}+V^{\perp}$, where $V^{\perp}=Id-V$. Note that \begin{center}
$R^*=R$, $RVR=\delta_1$ and $VR^{-1}=R^{-1}V$.
\end{center}

Now, consider $S:M_k\rightarrow M_k$ defined by $S(X)=\frac{1}{\lambda}RT(R^{-1}X(R^{-1})^*)R^*.$\\

Since $S^*(X)=\frac{1}{\lambda}(R^{-1})^*T^*(R^*XR)R^{-1}$ then $VS^*(V)V=$ 

$$=\frac{1}{\lambda}V(R^{-1})T^*(RVR)R^{-1}V=\frac{1}{\lambda}R^{-1}VT^*(\delta_1)VR^{-1}=\frac{\lambda}{\lambda}R^{-1}\delta_1R^{-1}=V.$$

Note that  $S(VM_kV)\subset VM_kV$ and $S|_{VM_kV}$ is irreducible, since $R^{-1}V=VR^{-1}$ and $T|_{VM_kV}$ is irreducible. 
Therefore, by Lemma \ref{lemma2},

\begin{center}
$S^*(V^{\perp}M_kV^{\perp})\subset V^{\perp}M_kV^{\perp}$ and $VS^*(\cdot)V|_{VM_kV}$ is irreducible.
\end{center}

Next, let $U$ be a unitary matrix such that $U^*VU=\begin{pmatrix}Id_{s\times s}& 0 \\ 
0& 0\end{pmatrix}$. 

Let $V_1=U^*VU$ and define $T_1:M_k\rightarrow M_k$ as $T_1(X)=U^*S(UXU^*)U$. \\

Since $T_1^*(X)=U^*S^*(UXU^*)U$ and $UV_1=VU$ then $V_1T^*_1(V_1)V_1=$

\begin{center}
$=V_1U^*S^*(UV_1U^*)UV_1=U^*VS^*(UV_1U^*)VU=U^*VS^*(V)VU=U^*VU=V_1.$
\end{center}

Thus, $V_1T_1^*(V_1M_kV_1)V_1\subset V_1M_kV_1$. Moreover,   $V_1T_1^*(\cdot)V_1:V_1M_kV_1\rightarrow V_1M_kV_1$ is irreducible, since $VS^*(\cdot)V|_{VM_kV}$ is irreducible and $V_1U^*=U^*V$.  \\

Next,
 $T_1^*(V_1^{\perp}M_kV_1^{\perp})\subset U^*S^*(UV_1^{\perp}M_kV_1^{\perp}U^*)U=$\begin{center}
$U^*S^*(V^{\perp}M_kV^{\perp})U\subset U^*V^{\perp}M_kV^{\perp}U= V_1^{\perp}M_kV_1^{\perp}$.
\end{center}

Hence, $T_1(V_1M_kV_1)\subset V_1M_kV_1$, by Lemma \ref{lemma2}. Since  $V_1T_1^*(\cdot)V_1|_{V_1M_kV_1}$ is irreducible then $T_1|_{V_1M_kV_1}$ is irreducible, by Lemma \ref{lemma2}.\\

Recall that $V_1T_1^*(V_1)V_1=V_1$ then, by Lemma \ref{lemma2},  the spectral radius of $V_1T_1^*(\cdot)V_1|_{V_1M_kV_1}$ is $1$. Thus, the spectral radius of its adjoint $T_1|_{V_1M_kV_1}$ is also $1$.\\

Now, let $W_1$ be a solution to the Problem \ref{problem} subjected to $T_1:M_k\rightarrow M_k$ and $V_1$. Thus,
\begin{itemize}
\item  rank$(W_1)=$ rank$(V_1)$,
\item  $\Im(V_1)\cap\ker(W_1)=\{\vec{0}\}$
\item $T^*_1(W_1M_kW_1)\subset W_1M_kW_1$
\item $T_1^*|_{W_1M_kW_1}$ is irreducible.
\end{itemize}

Next, define $Q=U^*R$. Note that $T_1(X)=\frac{1}{\lambda}QT(Q^{-1}X(Q^{-1})^*)Q^*$, \begin{center}
$T(X)=\lambda Q^{-1}T_1^*(QXQ^*)(Q^{-1})^*$ and $T^*(X)=\lambda Q^*T_1^*((Q^{-1})^*XQ^{-1})Q$.
\end{center}

Let $W$ be the orthogonal projection onto $\Im(Q^*W_1Q)$. Note that $T^*(WM_kW)\subset WM_kW$.\\

Now, since 
$WM_kW=Q^*W_1M_kW_1Q$ and $T_1^*|_{W_1M_kW_1}$ is irreducible then \begin{center}
$T^*(WM_kW)\subset WM_kW$, $T^*|_{WM_kW}$ is irreducible.
\end{center}

Since $\Im(V_1)=\Im(U^*\delta_1U)=\Im(U^*RVR^*U)=Im(QVQ^*)$ then $\Im(V)=\Im(Q^{-1}V_1(Q^{-1})^*)$.\\

Moreover, $\ker(W)=\ker(Q^*W_1Q)$.
Thus, $\ker(W)\cap \Im(V)=\{\vec{0}\}$, since $\ker(W_1)\cap \Im(V_1)=\{\vec{0}\}$.\\

Finally, rank$(W)=$ rank$(W_1)=$ rank$(V_1)=$ rank$(V).$
\end{proof}

\vspace{0.5 cm}

\begin{lemma}\label{solutionW}  Let $T_1:M_k\rightarrow M_k$ be a completely positive map such that $T_1(V_1M_kV_1)\subset V_1M_kV_1$, $T_1|_{V_1M_kV_1}$ is irreducible and $V_1T_1^*(V_1)V_1=V_1$, where $V_1=\begin{pmatrix}Id_{s\times s} & 0_{s\times k-s} \\ 
0_{k-s\times s} & 0_{k-s\times k-s}\end{pmatrix}$. The following statements are equivalent:
\begin{itemize}
\item[a)] $W$ is a solution to the Problem \ref{problem} subjected to these $T_1$ and $V_1$.
\item[b)] $W$ is an orthogonal projection onto $\Im\begin{pmatrix}Id_{s\times s}\\ 
S_{k-s\times s} \end{pmatrix}$, where $S\in M_{k-s\times s}(\mathbb{C})$\\ is a zero of
$f:M_{k-s\times s}(\mathbb{C})\rightarrow \mathbb{R}^+\cup\{0\}$, defined by \\

$f(X)=tr\left(T_1^*\left(\begin{pmatrix}Id& 0 \\ 
0& 0\end{pmatrix}\right)\begin{pmatrix}0 & 0 \\ 
0 & Id\end{pmatrix}\right)+tr\left(T_1^*\left(\begin{pmatrix}Id& 0 \\ 
0& 0\end{pmatrix}\right)\begin{pmatrix}X^*X & -X^* \\ 
-X & 0\end{pmatrix}\right)\\ $

$ +tr\left(T_1^*\left(\begin{pmatrix}0& X^* \\
X& XX^*\end{pmatrix}\right)\begin{pmatrix}0 & 0 \\ 
0 & Id\end{pmatrix}\right)+tr\left(T_1^*\left(\begin{pmatrix}0& X^* \\
X& 0\end{pmatrix}\right)\begin{pmatrix}0 & -X^* \\ 
-X & 0\end{pmatrix}\right).\\
$

\end{itemize}
\end{lemma}

\begin{proof} First, define $f:M_{k-s\times s}(\mathbb{C})\rightarrow \mathbb{R}^+\cup\{0\}$ as $$f(X)=tr\left(T_1^*\left(\begin{pmatrix}Id& X^* \\ 
X& XX^*\end{pmatrix}\right)\begin{pmatrix}X^*X & -X^* \\ 
-X & Id\end{pmatrix}\right).$$

Now, by Lemmas \ref{lemma1} and \ref{lemma2},  since $T_1(V_1M_kV_1)\subset V_1M_kV_1$
then 
\begin{center}
$T_1^*(V_1^{\perp}M_kV_1^{\perp})\subset V_1^{\perp}M_kV_1^{\perp}$ and $T_1^*(V_1^{\perp}M_k+M_kV_1^{\perp})\subset V_1^{\perp}M_k+M_kV_1^{\perp}.$

\end{center}

Therefore,
\begin{center}
$T_1^*\left(\begin{pmatrix}0 & 0 \\ 
0& XX^*\end{pmatrix}\right)=\begin{pmatrix}0 & 0 \\ 
0& B\end{pmatrix}$ and $T_1^*\left(\begin{pmatrix}0 & X^* \\ 
X& 0\end{pmatrix}\right)=\begin{pmatrix}0 & Y^* \\ 
Y& D\end{pmatrix}$.

\end{center}

Hence,
 $$tr\left(T_1^*\left(\begin{pmatrix}0& X^* \\
X& XX^*\end{pmatrix}\right)\begin{pmatrix}X^*X & -X^* \\ 
-X & 0\end{pmatrix}\right)=tr\left(\begin{pmatrix}0& Y^* \\
Y& D+B\end{pmatrix}\begin{pmatrix}X^*X & -X^* \\ 
-X & 0\end{pmatrix}\right)=$$

$$=tr\left(\begin{pmatrix}0& Y^* \\
Y& D\end{pmatrix}\begin{pmatrix}0 & -X^* \\ 
-X & 0\end{pmatrix}\right)=tr\left(T_1^*\left(\begin{pmatrix}0& X^* \\
X& 0\end{pmatrix}\right)\begin{pmatrix}0 & -X \\ 
-X^* & 0\end{pmatrix}\right).$$

This  equality yields the formula for $f(X)$ in the statement of this lemma. 
 
 Next, let us show the equivalence between $a)$ and $b)$.\\
 
\textit{Proof of }$a)\Rightarrow b): $  Let $W$ be a solution of Problem \ref{problem} subjected to  $T_1$ and $V_1$.

Since $T_{1}^*|_{WM_kW}$ is irreducible, there is $\delta\in P_k$ such that \begin{center}
$\Im(\delta)=\Im(W)$ and 
$T_1^*(\delta)=\lambda\delta$, 
\end{center} 

where $\lambda$ is the spectral radius of $T_{1}^*|_{WM_kW}$ by Perron-Frobenius theory \cite[Theorem 2.3]{evans}.

Now, $T_1^*(V_1^{\perp}M_k+M_kV_1^{\perp})\subset V_1^{\perp}M_k+M_kV_1^{\perp}$ implies that
\begin{center}
$V_1T_1^*(V_1^{\perp}\delta +V_1\delta V_1^{\perp})V_1=0$. 
\end{center}

Thus, 
$\lambda V_1\delta V_1=V_1T_1^*(\delta)V_1=V_1T_1^*(V_1^{\perp}\delta +V_1\delta V_1^{\perp})V_1+V_1T_1^*(V_1\delta V_1)V_1=V_1T_1^*(V_1\delta V_1)V_1$.\\

Since $\ker(\delta)\cap\Im(V_1)=\ker(W)\cap\Im(V_1)=\{\vec{0}\}$ then  
$V_1\delta V_1\in (P_k\cap  V_1M_kV_1)\setminus \{0\}$.

By Lemma \ref{lemma2}, $V_1T_1^*(\cdot)V_1|_{V_1M_kV_1}$ is irreducible. Thus,
$\Im(V_1\delta V_1)=\Im(V_1)$. \\

Next, since 
$V_1T_1^*(V_1\delta V_1)V_1=\lambda V_1\delta V_1$, $\Im(V_1\delta V_1)=\Im(V_1)$ and  
$V_1T_1^*(V_1)V_1= V_1$
then $\lambda=1$,

 by item $c)$ of Lemma \ref{lemma2}.
Moreover, the geometric multiplicity of $\lambda$   is $1$ (Lemma \ref{lemma3}). \\

Therefore, 
$V_1\delta V_1=\mu V_1$, for some $\mu>0$.
Thus, 
$\frac{\delta}{\mu}=\begin{pmatrix}Id& S^* \\
S& SS^*\end{pmatrix}$, for some $S\in M_{k-s\times s}(\mathbb{C})$.

So $T_1^*\left(\begin{pmatrix}Id& S^* \\
S& SS^*\end{pmatrix}\right)=\begin{pmatrix}Id& S^* \\
S& SS^*\end{pmatrix}$ and $W$ is the orthogonal projection onto $\Im\left(\begin{pmatrix}Id \\
S \end{pmatrix}\right)$. \\

Finally, note that $S$ is a zero of $f(X)$. The proof that $a)\Rightarrow b)$ is complete.\\

\textit{Proof of }$b)\Rightarrow a):$ Let $S\in M_{k-s\times s}(\mathbb{C})$ be a zero of $f(X)$.\\

Since $V_1T_1^*(V_1)V_1=V_1$ and $T_1^*(V_1^{\perp}M_k+M_kV_1^{\perp})\subset V_1^{\perp}M_k+M_kV_1^{\perp}$ then 

\begin{center}
$T_1^*\left(\begin{pmatrix}Id& S^* \\
S& SS^*\end{pmatrix}\right)=\begin{pmatrix}Id& Z^* \\
Z& R\end{pmatrix}\in P_k.$

\end{center}

Since $f(S)=0$ then  $$\Im\left(\begin{pmatrix}S^*S & -S^* \\ 
-S & Id\end{pmatrix}\right)=\Im\left(\begin{pmatrix}S^* \\ 
-Id\end{pmatrix}\right)\subset\ker\left(\begin{pmatrix}Id& Z^* \\
Z& R\end{pmatrix}\right),$$ 

which is equivalent to $S^*=Z^*$ and $ZS^*=R$.
Thus, \begin{center}
$T_1^*\left(\begin{pmatrix}Id& S^* \\
S& SS^*\end{pmatrix}\right)=\begin{pmatrix}Id& S^* \\
S& SS^*\end{pmatrix}.$

\end{center}


Let $W_1$ be the orthogonal projection onto $\Im\left(\begin{pmatrix}Id& S^* \\
S& SS^*\end{pmatrix}\right)$.  

Note that
\begin{itemize}
\item $T_1^*(W_1M_kW_1)\subset W_1M_kW_1$,
\item $T_1^*|_{W_1M_kW_1}$ has spectral radius equals to 1 (by item $c)$ of Lemma \ref{lemma2}),
\item $\Im(W_1)\cap\ker(V_1)=\{\vec{0}\}$,
\item $\ker(W_1)\cap\Im(V_1)=\{\vec{0}\}$.
\end{itemize}

In order to complete this proof, we must show that $T_1^*|_{W_1M_kW_1}$ is irreducible. 
If this is not the case, then 
there is $\delta_1\in P_k\cap  W_1M_kW_1$ such that $T_1^*(\delta_1)=\alpha\delta_1$, $\alpha>0$ and $0<$ rank$(\delta_1)<$ rank$(W_1)$.\\

Next, if $V_1\delta_1V_1=0$ then $\Im(\delta_1)\subset\ker(V_1)$. Since $\Im(\delta_1)\subset \Im(W_1)$ then $\Im(W_1)\cap\ker(V_1)\neq\{\vec{0}\}$, which is a contradiction.
So $V_1\delta_1V_1\neq 0$.\\

Repeating the same argument used previously in $[a)\Rightarrow b)]$, we have 
$V_1T_1^*(V_1\delta_1 V_1)V_1=\alpha V_1\delta_1 V_1.$\\

Since rank$(V_1\delta_1V_1)\leq$ rank$(\delta_1)<$ rank$(W_1)=$ rank$(V_1)$ then $V_1T_1^*(\cdot)V_1|_{V_1M_kV_1}$ is not irreducible, which is a contradiction with Lemma \ref{lemma2}. \\

Finally, $T_1^*|_{W_1M_kW_1}$ is irreducible, $\ker(W_1)\cap\Im(V_1)=\{\vec{0}\}$ and rank$(W_1)=$ rank$(V_1)$.
\end{proof}

\vspace{0.5 cm}

\begin{remark} \label{remarkquadratic}

Finding a zero for $f(X)$ is an unconstrained quadratic minimization problem, if $M_{k-s\times s}(\mathbb{C})$ is regarded as a real vector space. We are only interested in this zero when its uniqueness is granted. The uniqueness occurs only when the real symmetric matrix associated to the bilinear form $g: M_{k-s\times s}(\mathbb{C})\times M_{k-s\times s}(\mathbb{C})\rightarrow\mathbb{R}$  is positive definite, where  $g(X,Y)\text{ is the real part of }$ \\
$$tr\left(T_1^*\left(\begin{pmatrix}Id& 0 \\ 
0& 0\end{pmatrix}\right)\begin{pmatrix}X^*Y & 0 \\ 
0 & 0\end{pmatrix}+T_1^*\left(\begin{pmatrix}0& 0 \\
0 & XY^*\end{pmatrix}\right)\begin{pmatrix}0 & 0 \\ 
0 & Id\end{pmatrix}+T_1^*\left(\begin{pmatrix}0& X^* \\
X& 0\end{pmatrix}\right)\begin{pmatrix}0 & -Y^* \\ 
-Y & 0\end{pmatrix}\right).\\
$$
\end{remark}

\section{Algorithms}
In this section, we bring all the results together in our algorithms. Our main algorithm  (Algorithm 3) checks whether a PPT state $B\in M_k\otimes M_k$ with a vector  $v\in\Im(B)$ with tensor rank $k$ can be put in the filter normal form or not. It searches for all pairs of orthogonal projections $(V,W)$ as described in problem \ref{problem} for a positive map $T:M_k\rightarrow M_k$ equivalent to $G_B((\cdot)^t):M_k\rightarrow M_k$ . This  procedure reproduces the proof of the theorem \ref{theoremprincipal}, particularly the part $(3\Rightarrow 2)$ .

\vspace{0.5 cm}
\noindent\textbf{Algorithm 1:} Given a completely positive map $T:VM_kV\rightarrow VM_kV$, this algorithm finds $V_1M_kV_1\subset VM_kV$ left invariant by $T:VM_kV\rightarrow VM_kV$ such that $T|_{V_1M_kV_1}$ is irreducible. Note that every time $V$ is redefined its rank decreases. So the process shall stop.  \\

\begin{itemize}
\item[Step 1:] Compute rank$(V)$.\\\\
$\bullet$ If rank$(V)=1$ then define $V_1=V$.\\
$\bullet$ If rank$(V)\neq 1$ then do Step 2.\\

\item[Step 2:] Find the spectral radius $\lambda$ of $T:VM_kV\rightarrow VM_kV$, compute $\dim(\ker(T-\lambda Id|_{VM_kV})))$ and find a Perron eigenvector $\gamma\in VM_kV$ associated to $\lambda$.\\\\
$\bullet$ If $\dim(\ker(T-\lambda Id|_{VM_kV})))=1$ and $\Im(\gamma)=\Im(V)$ then do Step 3.\\
$\bullet$ If $\dim(\ker(T-\lambda Id|_{VM_kV})))=1$  and $\Im(\gamma)\neq\Im(V)$ then redefine V as the orthogonal projection onto $\Im(\gamma)$ and do Step 3.\\
$\bullet$ If $\dim(\ker(T-\lambda Id|_{VM_kV})))\neq 1$ then find an Hermitian matrix $\gamma'\in VM_kV$  and $\epsilon>0$ such that $T(\gamma')=\lambda\gamma'$, $\gamma-\epsilon\gamma'\in P_k\setminus\{0\}$ and rank$(\gamma-\epsilon\gamma')<$rank$(\gamma)$. Redefine $V$ as the orthogonal projection onto $\Im(\gamma-\epsilon\gamma')$ and repeat Step 2.\\

\item[Step 3:] Find a Perron eigenvector $\delta\in VM_kV$ of $VT^*(\cdot)V:VM_kV\rightarrow VM_kV$ associated to $\lambda$.\\\\
$\bullet$ If $\Im(\delta)=\Im(\gamma)$ then $T:VM_kV\rightarrow VM_kV$ is irreducible. Define $V_1=V$.\\
$\bullet$ If $\Im(\delta)\neq\Im(\gamma)$ then redefine $V$ as the orthogonal projection onto $\Im(V)\cap\ker(\delta)$ and return to Step 1.\\

\end{itemize}

\noindent\textbf{Algorithm 2:} This algorithm finds the unique solution $W$ of Problem \ref{problem} subjected to $T$ and $V$.\\

\begin{itemize}
\item[Step 1:] Find the spectral radius $\lambda$ of $T:VM_kV\rightarrow VM_kV$ and the invertible matrix $Q\in M_k$ of Lemma \ref{lemmaQ}.\\

\item[Step 2:] Define $T_1:VM_kV\rightarrow VM_kV$ as $T_1(X)=\frac{1}{\lambda}QT(Q^{-1}X(Q^{-1})^*)Q^*$. Regard  $M_{k-s\times s}(\mathbb{C})$ as a real vector space and find the unique zero $S\in M_{k-s\times s}$ of the quadratic function $f:M_{k-s\times s}(\mathbb{C})\rightarrow \mathbb{R}^+\cup\{0\}$ defined in Lemma \ref{solutionW}.\\
$\bullet$ If the unique zero exists then define $W$ as the orthogonal projection onto $\Im\left(Q^*\begin{pmatrix}Id_{s\times s}\\ 
S_{k-s\times s} \end{pmatrix}\right)$\\
$\bullet$ If the zero does not exist or it is not unique then there is no such $W$.\\\\
\end{itemize}

\noindent\textbf{Algorithm 3:} Given a PPT matrix $B\in M_k\otimes M_k$ and a vector $v\in\Im(B)$ with tensor rank $k$, this algorithm checks whether $G_B((\cdot)^t):M_k\rightarrow M_k$ is equivalent to a doubly stochastic map.\\

Find the invertible matrix $P\in M_k$ of Lemma \ref{lemmaP}. Let $A=(P\otimes Id)B(P\otimes Id)^*$ and $T:M_k\rightarrow M_k$ be $T(X)=G_A((\cdot)^t)$. In order to start the procedure set $V'=Id$.\\

\begin{itemize}
\item[Step 1:] Find $VM_kV\subset V'M_kV'$ such that $T(VM_kV)\subset VM_kV$ and $T|_{VM_kV}$ is irreducible via algorithm 1.\\\\
$\bullet$ If $V=V'$ then $T$ is equivalent to a doubly stochastic map and also $G_B((\cdot)^t):M_k\rightarrow M_k$.\\
$\bullet$ If $V\neq V'$ then do Step 2.\\

\item[Step 2:] Search for the unique solution $W$ of Problem \ref{problem} subjected to $T$ and $V$ via algorithm 2.\\\\
$\bullet$ If there is no such $W$ then $T$ is not equivalent to a doubly stochastic map and neither is $G_B((\cdot)^t):M_k\rightarrow M_k$.\\
$\bullet$ If there is such $W$ then find an invertible matrix $R\in M_k$ such that\begin{center}
 
$R(Id-V'+V)=(Id-V'+V)$  and $\Im(R(V'-W))=\Im(V'-V)$.
\end{center} 

Redefine $T$ as $RT(R^{-1}X(R^{-1})^*)R^*$ and  $V'$ as $V'-V$ then repeat Step 1.

\end{itemize}

\vspace{0.5 cm}

\begin{remark}  If the algorithm finds out that $T:M_k\rightarrow M_k$ is equivalent to a doubly stochastic map then 
the $s$ values attained by $V'$ $($in this run$)$ are orthogonal projections $V_1,\ldots V_s$ 
such that
\begin{tabbing}
\hspace{8 cm}\=\kill
 \hspace{3 cm}$V_iV_j=0$ for $i\neq j$,  \> $\sum_{i=1}^sV_i=Id$, \\ 
\hspace{3 cm} $T'(V_iM_kV_i)\subset V_iM_kV_i$, \> $T'|_{V_iM_kV_i}$ is irreducible
 for every $1\leq i\leq s$,
\end{tabbing} 
 where $T'(X)$ is the last value attained by $T$.  \\
 
 Recall that $T'(X)=QG_A((Q^{-1}X(Q^{-1})^*)^t)Q^*$ for some invertible $Q\in M_k$. Thus, $T'(X)=G_C(X^t)$, where $C=((Q^{-1})^t\otimes Q)A((Q^{-1})^t\otimes Q)^*$. 
By Proposition \ref{propkey},
$C=\sum_{i=1}^s C_i$, where $C_i=(V_i^t\otimes V_i)C(V_i^t\otimes V_i)$.\\

Note that $\Im (T'(X))\supset\Im(X)$, for every $X\in P_k$, and $T'|_{V_iM_kV_i}=G_{C_i}((\cdot)^t)|_{V_iM_kV_i}$ is irreducible. Therefore, $G_{C_i}((\cdot)^t)|_{V_iM_kV_i}$ is fully indecomposable. So the scaling algorithm \cite{gurvits2003, gurvits2004, cariellosink} applied to $G_{C_i}((\cdot)^t)|_{V_iM_kV_i}$ converges to  a doubly stochastic map $G_{D_i}((\cdot)^t)|_{V_iM_kV_i}$. Note that $D_i$ is the filter normal form of $C_i$.
\end{remark}

\section{A simple extension to $M_k\otimes M_m$}

Recall the identification $M_m\otimes M_k \simeq M_{mk}$. Let $B=\sum_{i=1}^nC_i\otimes D_i\in M_k\otimes M_m$ be a PPT matrix and define $\widetilde{B}\in M_{mk}\otimes M_{mk}$ as $$\widetilde{B}=\sum_{i=1}^n (Id_{m}\otimes C_i)\otimes (D_i\otimes Id_k).$$ 
The next result is the key to extend algorithm $3$ to PPT matrices in $M_k\otimes M_m$.
\begin{lemma}\label{lemmaextension} Let $B=\sum_{i=1}^nC_i\otimes D_i\in M_k\otimes M_m$ be a PPT matrix. Consider the PPT matrix $\widetilde{B}\in M_{mk}\otimes M_{mk}$ as defined above. Then, $G_B((\cdot)^t): M_k\rightarrow M_m$ is equivalent to a doubly stochastic map if and only if $G_{\widetilde{B}}((\cdot)^t): M_{mk}\rightarrow M_{mk}$ is equivalent to a doubly stochastic map.
\end{lemma}
\begin{proof}
Let $e_1,\ldots,e_m$ be the canonical basis of $\mathbb{C}^m$. Note that for every $\sum_{i,j=1}^m e_ie_j^t\otimes B_{ij}\in M_m\otimes M_k$, $$G_{\widetilde{B}}(\sum_{i,j=1}^m e_je_i^t\otimes B_{ij}^t)=G_B(\sum_{i=1}^m B_{ii}^t)\otimes Id_{ k}.$$

Therefore, by \cite[Corollary 3.5]{cariellosink}, $G_B((\cdot)^t): M_k\rightarrow M_m$ is equivalent to a doubly stochastic map if and only if $G_{\widetilde{B}}((\cdot)^t): M_{mk}\rightarrow M_{mk}$ is equivalent to a doubly stochastic map.
\end{proof}

\begin{corollary} Let $B\in M_k\otimes M_m$ and $\widetilde{B}\in M_{mk}\otimes M_{mk}$ be as in Lemma \ref{lemmaextension}. Then, $B$ can be put in the filter normal form if and only if $\widetilde{B}$ can be put in the filter normal form.
\end{corollary}
Thus, if there is a vector $v \in \Im(\widetilde{B})$ with tensor rank $mk$ then we can run algorithm $3$ with $\widetilde{B}$. If the algorithm finds out that $G_{\widetilde{B}}((\cdot)^t): M_{mk}\rightarrow M_{mk}$ is equivalent to a doubly stochastic map then $B$ can be put in the filter normal form.

\section*{Summary and Conclusion}

In this work  we described a procedure to determine whether $G_A:M_k\rightarrow M_k$, $G_A(X)=\sum_{i=1}^n tr(A_iX)B_i$ is equivalent to a doubly stochastic map or not when $A=\sum_{i=1}^nA_i\otimes B_i \in M_k\otimes M_k\simeq M_{k^2}$ is a PPT matrix and there is a full tensor rank vector within its image. The difficult  part of the algorithm is finding Perron eigenvectors of $G_A((\cdot)^t):M_k\rightarrow M_k$. 

This procedure can be used to determine whether a PPT matrix can be put in the filter normal form, which is a very useful tool to study entanglement of quantum mixed states. The existence of a full tensor rank vector within its range  is a necessary condition to put a separable matrix in the filter normal form.

Finally, we noticed that the procedure can be extended to PPT matrices in  $M_k\otimes M_m$.


\section*{Appendix}

\begin{proof}[Proof of Lemma \ref{lemma1}]

By Definition \ref{defcompletepositive}, $T(X)=\sum_{i=1}^mA_iXA_i^*.$ 

Now, if $T(V)=\sum_{i=1}^mA_iVA_i^*\in VM_kV$ then 
$A_iVA_i^*\in VM_kV$, for every $i$.

Since $A_iVA_i^*=A_iV(A_iV)^*$ then $A_iV\in VM_k$. Thus,  $VA_i^*\in M_kV$, for every $i$.

Hence, 
$A_i(VX+YV)A_i^*\in VM_k+M_kV$, for every $X,Y\in M_k$.  
\end{proof}
\vspace{0,5cm}

\begin{proof}[Proof of Lemma \ref{lemma2}] 
a) Note that
$$0=tr(T(V_1)(V-V_1))=tr(V_1T^*(V-V_1))=tr(VV_1VT^*(V-V_1))=tr(V_1VT^*(V-V_1)V).$$

Thus, $\Im(VT^*(V-V_1)V)\subset\Im(V-V_1)$ and $$VT^*((V-V_1)M_k(V-V_1))V\subset (V-V_1)M_k(V-V_1).$$

\noindent
b) If $T|_{VM_kV}$ is not irreducible, neither is $VT^*(\cdot)V|_{VM_kV}$ by item a).\\

Now, if $VT^*(\cdot)V:VM_kV\rightarrow VM_kV$ is not irreducible then there is a proper subalgebra $V_1M_kV_1\subset VM_kV$ such that 
$VT^*(V_1M_kV_1)V\subset V_1M_kV_1$ and $V_1\neq 0$.\\

Hence, $0=tr(VT^*(V_1)V(V-V_1))=tr(T^*(V_1)V(V-V_1)V)=tr(V_1T(V-V_1))$. Thus, $$T((V-V_1)M_k(V-V_1))\subset (V-V_1)M_k(V-V_1).$$

\noindent
c) Let $\gamma\in P_k\cap VM_kV$ be such that $\gamma^2=\delta$. Denote by $\gamma^{+}$ the  Hermitian pseudo-inverse of $\gamma$. So $\gamma^{+}\gamma=\gamma\gamma^{+}=V$. 

Note that $\frac{1}{\lambda}\gamma^{+}T(\gamma V\gamma)\gamma^{+}=V$. Therefore, the operator norm of \begin{center}
$\frac{1}{\lambda}\gamma^{+}T(\gamma (\cdot) \gamma)\gamma^{+}:VM_kV\rightarrow VM_kV$
\end{center} induced by the spectral norm on $M_k$ is 1 \cite[Theorem 2.3.7]{Bhatia1}. \\

Let $Y\in VM_kV$ be an eigenvector of $T:VM_kV\rightarrow VM_kV$ associated to $\alpha$ such that the spectral norm of $\gamma^{+}Y\gamma^{+}$ is 1. Then, the spectral norm of  $\frac{1}{\lambda}\gamma^{+}T(\gamma \gamma^{+}Y\gamma^{+} \gamma)\gamma^{+}=\frac{\alpha}{\lambda}\gamma^{+}Y\gamma^{+}$ is smaller or equal to the operator norm of $\frac{1}{\lambda}\gamma^{+}T(\gamma (\cdot) \gamma)\gamma^{+}:VM_kV\rightarrow VM_kV$. Thus, $\frac{|\alpha|}{\lambda}\leq 1$. 
\end{proof}
\hspace{1cm}
\begin{proof}[Proof of Lemma \ref{lemma3}]
Since $VT^*(\cdot)V|_{VM_kV}$ and $T|_{VM_kV}$ are adjoint with respect to the trace inner product and $\lambda\in \mathbb{R}$ then $\text{rank}(VT^*(\cdot)V-\lambda Id|_{VM_kV})=       
\text{rank}(T-\lambda Id|_{VM_kV})$. So the geometric multiplicity of $\lambda$ is the same for both maps.\\

Moreover, $T|_{VM_kV}$ is irreducible if and only if $VT^*(\cdot)V|_{VM_kV}$ is irreducible by Lemma \ref{lemma2}. item $b)$. So conditions $1)$ and $2)$ are necessary for irreducibility by Perron-Frobenius theory (\cite[Theorem 2.3]{evans}).\\

Next, let us assume by contradiction that conditions $1)$ and $2)$ hold and $T|_{VM_kV}$ is not irreducible.

Thus, there is an orthogonal projection $V_1\in M_k$ such that $T(V_1M_kV_1)\subset V_1M_kV_1$, $V_1V=VV_1=V_1$ and $V-V_1\neq 0$.\\

By item $a)$ of Lemma \ref{lemma2}, $VT^*((V-V_1)M_k(V-V_1))V\subset (V-V_1)M_k(V-V_1)$. Since $T^*:M_k\rightarrow M_k$ is completely positive then $VT^*(\cdot)V:M_k\rightarrow M_k$ is too. \\

Hence, by Lemma \ref{lemma1},
\begin{equation}\label{eq=2}
VT^*((V-V_1)M_k+M_k(V-V_1))V\subset (V-V_1)M_k+M_k(V-V_1).
\end{equation}

Now, since $V_1V=V_1$ and $\lambda\delta=VT^*(\delta)V=VT^*((V-V_1)\delta+V_1\delta (V-V_1))V+VT^*(V_1\delta V_1)V$ then, by Equation \ref{eq=2},
$\lambda V_1\delta V_1=V_1T^*(V_1\delta V_1)V_1.$  

Note that $V_1\delta V_1\neq 0$, since $\Im(V_1)\subset \Im(V)=\Im(\delta)$. Hence, $\lambda$ is an eigenvalue of $V_1T^*(\cdot)V_1: V_1M_kV_1\rightarrow V_1M_kV_1$. Therefore, it is also an eigenvalue of its adjoint $T:V_1M_kV_1\rightarrow V_1M_kV_1$. \\

Actually, $\lambda$ is the spectral radius of $T|_{V_1M_kV_1}$, since $V_1M_kV_1\subset VM_kV$ and by hypothesis. So, by Perron-Frobenius theory, there is $\gamma'\in (V_1M_kV_1 \cap P_k)\setminus \{0\}$ such that $T(\gamma')=\lambda\gamma'$.\\

Note that $\gamma$ and $\gamma'$ are linearly independent, since $\Im(\gamma')\subset\Im(V_1)\neq\Im(V)=\Im(\gamma)$. Thus, the geometric multiplicity of $\lambda$ is not 1. Absurd!
\end{proof}

\end{document}